\pdfoutput=1
\RequirePackage{ifpdf}
\ifpdf 
\documentclass[pdftex]{sigma}
\else
\documentclass{sigma}
\fi

\numberwithin{equation}{section}

\newtheorem{Theorem}{Theorem}[section]
\newtheorem{Corollary}[Theorem]{Corollary}
\newtheorem{Lemma}[Theorem]{Lemma}
\newtheorem{Proposition}[Theorem]{Proposition}
{ \theoremstyle{definition}
\newtheorem{Definition}[Theorem]{Definition}
\newtheorem{Example}[Theorem]{Example}
\newtheorem{Remark}[Theorem]{Remark} }

\begin{document}
\allowdisplaybreaks

\newcommand{\arXivNumber}{2003.14104}

\renewcommand{\thefootnote}{}

\renewcommand{\PaperNumber}{092}

\FirstPageHeading

\ShortArticleName{On the Generalized Cluster Algebras of Geometric Type}

\ArticleName{On the Generalized Cluster Algebras\\ of Geometric Type\footnote{This paper is a~contribution to the Special Issue on Cluster Algebras. The full collection is available at \href{https://www.emis.de/journals/SIGMA/cluster-algebras.html}{https://www.emis.de/journals/SIGMA/cluster-algebras.html}}}

\Author{Liqian BAI~$^{\dag^1}$, Xueqing CHEN~$^{\dag^2}$, Ming DING~$^{\dag^3}$ and Fan XU~$^{\dag^4}$}

\AuthorNameForHeading{L.~Bai, X.~Chen, M.~Ding and F.~Xu}

\Address{$^{\dag^1}$~School of Mathematics and Statistics, Northwestern Polytechnical University,\\
\hphantom{$^{\dag^1}$}~Xi'an, Shaanxi 710072, P.~R.~China}
\EmailDD{\href{mailto:bailiqian@nwpu.edu.cn}{bailiqian@nwpu.edu.cn}}

\Address{$^{\dag^2}$~Department of Mathematics, University of Wisconsin-Whitewater,\\
\hphantom{$^{\dag^2}$}~800 West Main Street, Whitewater, WI~53190, USA}
\EmailDD{\href{mailto:chenx@uww.edu}{chenx@uww.edu}}

\Address{$^{\dag^3}$~School of Mathematics and Information Science, Guangzhou University,\\
\hphantom{$^{\dag^3}$}~Guangzhou 510006, P.~R.~China}
\EmailDD{\href{mailto:m-ding04@mails.tsinghua.edu.cn}{m-ding04@mails.tsinghua.edu.cn}}

\Address{$^{\dag^4}$~Department of Mathematical Sciences, Tsinghua University, Beijing 100084, P.~R.~China}
\EmailDD{\href{mailto:fanxu@mail.tsinghua.edu.cn}{fanxu@mail.tsinghua.edu.cn}}

\ArticleDates{Received April 01, 2020, in final form September 14, 2020; Published online September 28, 2020}

\Abstract{We develop and prove the analogs of~some results shown in [Berenstein~A., Fomin~S., Zelevinsky~A., \textit{Duke Math.~J.} \textbf{126} (2005), 1--52] concerning lower and upper bounds of~cluster algebras to the generalized cluster algebras of~geometric type. We show that lower bounds coincide with upper bounds under the conditions of~acyclicity and coprimality. Consequently, we obtain the standard monomial bases of~these generalized cluster algebras. Moreover, in the appendix, we prove that an~acyclic generalized cluster algebra is~equal to the corresponding generalized upper cluster algebra without the assumption of~the existence of~coprimality.}


\Classification{13F60; 05E16}

\begin{flushright}
\begin{minipage}{65mm}
\it Dedicated to our teacher Jie Xiao\\ on the occasion of his sixtieth birthday
\end{minipage}
\end{flushright}

\renewcommand{\thefootnote}{\arabic{footnote}}
\setcounter{footnote}{0}

\section{Background}
Fomin and Zelevinsky invented the concept of~cluster algebras \cite{ca1,ca2} in order
to create an~algebraic framework for studying total positivity in algebraic groups and canonical bases in quantum groups. As a natural generalization,
Chekhov and Shapiro introduced the generalized cluster algebras which arise from the Teichm\"{u}ller spaces of~Riemann surfaces with orbifold points \cite{CS}.
The main difference between cluster algebras and generalized cluster algebras is that the binomial exchange relations for cluster
variables of~cluster algebras are replaced by the multinomial exchange relations for those cluster
variables of~generalized cluster algebras. In \cite{CS}, Chekhov and Shapiro have shown that the generalized cluster algebras possess the remarkable Laurent phenomenon. Many other important properties of~cluster algebras are also shown to hold in the generalized cluster algebras such as finite type classification, $g$-vectors and $F$-polynomials~\cite{CS, nak2}.

Motivated by the Laurent phenomenon established in \cite{ca1}, Berenstein, Fomin and Zelevinsky in \cite{bfz} introduced the notion of~an~upper cluster algebra, which is a certain (maybe infinitely many) intersection of~Laurent polynomial rings.
The upper bound is a certain finite intersection of~rings of~Laurent polynomials.
The lower bound is the subalgebra of~a cluster algebra generated by specific finitely many cluster variables.
These four algebras satisfy the following relations:
\begin{gather*}
 \text{lower bound}\subseteq \text{cluster algebra}\subseteq \text{upper cluster algebra}\subseteq \text{upper bound}.
\end{gather*}
They proved that for a cluster algebra possessing an~acyclic and coprime seed, its lower bound coincides with its upper bound, thus, all four algebras coincide. The standard monomial bases of~these kinds of~cluster algebras can then be naturally constructed.
Beyond acyclic cluster algebras, Muller~\cite{muller13} introduced the notion of~locally acyclic cluster algebras, and it turns out that locally acyclic cluster algebras coincide with their upper cluster algebras~\cite{muller13, muller14}. Bucher, Machacek and Shapiro~\cite{BMS} discussed how the choice of~the ground ring impacts whether a cluster algebra is equal to its upper cluster algebra.

From a geometric point of~view, upper cluster algebras are more natural than cluster algebras~\cite{BMRS, GHK1}. In~\cite{MM15}, Matherne and Muller provided techniques for producing explicit presentations of~upper cluster algebras. Plamondon~\cite{P} used quiver representations to obtain a formula for certain elements of~skew-symmetric upper cluster algebras. Lee, Li and Mills~\cite{LLM} developed an~elementary formula for certain non-trivial elements of~an~upper cluster algebra with positive coefficients, and proved that these elements form a basis of~an~acyclic cluster algebra.

Gekhtman, Shapiro and Vainshtein~\cite{gsv18} proved that generalized upper cluster algebras over certain rings retain all properties of~ordinary upper cluster algebras, and under certain copri\-mality conditions coincide with the intersection of~rings of~Laurent polynomials in a finite collection of~clusters.

The aim of~this paper is to continue the investigation of~the structure of~generalized cluster algebras. By using the methods developed in \cite{bfz}, we prove that the conditions of~acyclicity and coprimality close the gap between lower bounds and upper bounds associated to generalized cluster algebras as the extension of~the similar results of~ordinary cluster algebras. Consequently, we obtain the standard monomial bases of~these generalized cluster algebras. It would be desi\-rable to apply the results that we obtain in this paper to construct some good bases of~the corresponding generalized cluster algebras in the future work.

In the appendix, we extend Muller's results on acyclic and locally acyclic cluster algebras to~acyclic generalized cluster algebras. By using the same arguments as given in~\cite{muller14}, we prove that acyclic generalized cluster algebras coincide with their generalized upper cluster algebras without the assumption of~the existence of~coprimality.

\section{Preliminaries}

First, let us recall the definition of~the generalized cluster algebras of~geometric type (see \cite{CS, gsv18}).

In the following, we use $[i,j]$ to denote the set $\{i,i+1,\ldots,j-1,j\}$ for integers $i<j$. Let~$m$ and $n$ be positive integers with $m\geq n$.
An $n\times n$ matrix $B$ is called skew-symmetrizable
if there exists a diagonal matrix $D=\mathop{\rm diag}\big(\widetilde{d}_1,\widetilde{d}_2,\ldots,\widetilde{d}_n\big)$, where $\widetilde{d}_i$ are positive integers for all $i\in[1,n]$, such that $DB$ is skew-symmetric.
Let~${\rm\mathbf{ex}}$ be a subset of~$[1,m]$ with $|{\rm\mathbf{ex}}|=n$.
Let~$\widetilde{B}=(b_{ij})$ be an~$m\times n$ integer matrix such that $\widetilde{B}$ has the $n\times n$ skew-symmetrizable submatrix $B$ with rows labeled by ${\rm\mathbf{ex}}$.
The columns of~$\widetilde{B}$ are also labeled by $\mathbf{ex}$.
The matrix $B$ is called the principal part of~$\widetilde{B}$.
For each $s \in \mathbf{ex}$, let $d_s$ be a positive integer such that $d_s$ divides $b_{is}$ for any $i\in[1,m]$, and let $\beta_{is}:=\frac{b_{is}}{d_s}$.

Let~us denote by $\mathbb{Q}(x_{1},x_2,\ldots,x_{m})$ the function field of~$m$ variables over $\mathbb{Q}$ with a transcendence basis $\{x_1,x_2,\ldots,x_m\}$. Let~the coefficient group $\mathbb{P}$ be the multiplicative free abelian group generated by $\{x_{i}\,|\, i\in[1,m]-{\rm\mathbf{ex}} \}$ and $\mathop{\mathbb{ZP}}$ be its integer group ring.
For each $i\in{\rm\mathbf{ex}}$, we will denote by $\rho_i:=\{\rho_{i,0},\ldots,\rho_{i,d_i}\}$ the $i$th string, where $\rho_{i,0}=\rho_{i,d_i}=1$ and $\rho_{i,j}$ are monomials in~$\mathbb{Z} [x_{i}\,|\, i\in{[1,m]-\rm\mathbf{ex}}]$ for $1\leq j\leq d_i-1$.

\begin{Definition}\label{def of generalized seed}
A~generalized seed (of geometric type) is a triple $\big(\widetilde{\mathbf{x}},\rho,\widetilde{B}\big)$, where
\begin{itemize}\itemsep=0pt
 \item[(1)] the set $\widetilde{\mathbf{x}}=\{x_{1},x_2,\ldots,x_{m}\}$ is~called the extended cluster and the set $\mathbf{x}=\{x_{i}\,|\, i\in{\rm\mathbf{ex}}\}$ is~called the cluster
 whose elements are called cluster variables, and the elements of~$\{x_{i}\,|\, i\in[1,m]-{\rm\mathbf{ex}}\}$ are called frozen variables;
 \item[(2)] $\rho=\{\rho_{i}\,|\, i\in{\rm\mathbf{ex}}\}$ is the set of~strings;
 \item[(3)] the matrices $\widetilde{B}$ and $B$ are called the extended exchange matrix and the exchange matrix, respectively.
\end{itemize}
\end{Definition}

Define the function $[x]_+:=x$ if $x\geq0$, and $[x]_+:=0$ otherwise.

\begin{Definition}
For $i\in{\rm\mathbf{ex}}$, the mutation of~a generalized seed $\big(\widetilde{\mathbf{x}},\rho,\widetilde{B}\big)$ in direction $i$ is another generalized seed \smash{$\mu_i\big(\widetilde{\mathbf{x}},\rho,\widetilde{B}\big):=\big(\widetilde{\mathbf{x}}',
\rho',\widetilde{B}'\big)$}, where
\begin{itemize}\itemsep=0pt
\item[(1)]the set $\widetilde{\mathbf{x}}':=(\widetilde{\mathbf{x}}- \{x_i\})\cup\{x'_{i}\}$ with
 \begin{gather*}
 x_{i}':=\mu_i(x_i)=x_i^{-1}\left(\sum\limits_{r=0}^{d_i}\rho_{i,r}\prod\limits_{j=1}^{m} x_{j}^{r[\beta_{ji}]_+ +(d_i-r)[-\beta_{ji}]_+}\right)\!,
 \end{gather*}
 which is called the exchange relation;
\item[(2)]$\rho' :=\mu_i(\rho)=\big\{\rho'_{i},\rho_{j}\,|\, j\in{\rm\mathbf{ex}}-\{i\}\big \}$, where $\rho_{i}'=\{\rho'_{i,0},\ldots,\rho'_{i,d_i}\}$ such that $\rho'_{i,r}=\rho_{i,d_i-r}$ for~$0\leq r\leq d_i$;

\item[(3)]
the matrix $\widetilde{B}':=\mu_i\big(\widetilde{B}\big)$ is defined by
\begin{gather*}
b'_{kl}=
 \begin{cases}
 -b_{kl} &\text{if}\quad k=i\quad \text{or}\quad l=i,
 \\
b_{kl}+([b_{il}]_+b_{ki}+ b_{il}[-b_{ki}]_+)&\text{otherwise}.
 \end{cases}
\end{gather*}

\end{itemize}
\end{Definition}

Note that $\mu_i$ is an~involution. The generalized seed $\big(\widetilde{\mathbf{y}},\varepsilon,\widetilde{A}\big)$ is said to be mutation-equivalent to \smash{$\big(\widetilde{\mathbf{x}},\rho,\widetilde{B}\big)$},
if $\smash{\big(\widetilde{\mathbf{y}},\varepsilon,\widetilde{A}\big)}$ can be obtained from
$\smash{\big(\widetilde{\mathbf{x}},\rho,\widetilde{B}\big)}$ by a sequence of~seed mutations. If $d_r=1$ for all $r \in {\rm\mathbf{ex}}$, then one recovers the ordinary cluster algebras.

\begin{Example}[{\cite[Theorem 2.7]{CS}\label{example1}}]
Let~$\mathbf{\widetilde{x}}=\mathbf{x}=\{x_1,x_2\}$ and
$\widetilde{B}=B=
 \left(\begin{smallmatrix} \phantom{-}0 & 1 \\ -2 & 0 \end{smallmatrix}\right)$.
Let~$(d_1,d_2)=(2,1)$ and $\rho=\{\rho_1,\rho_2\}$, where $\rho_1=\{1,h,1\}$ for any $h\in\mathbb{Z}$ and $\rho_2=\{1,1\}$.
The triple $(\mathbf{x},\rho,B)$ is a generalized seed. Let~$B'=
 \left(\begin{smallmatrix} 0 & -1 \\ 2 & \phantom{-}0 \end{smallmatrix}\right)$.
Then we have
\begin{gather*}
 \cdots\stackrel{\mu_1}\longleftrightarrow(\{x_1,x_0\},\rho,B')
\stackrel{\mu_2}\longleftrightarrow (\{x_1,x_2\},\rho,B)
\stackrel{\mu_1}\longleftrightarrow(\{x_3,x_2\},\rho,B')
\\ \phantom{\cdots}
 {}\stackrel{\mu_2}\longleftrightarrow (\{x_3,x_4\},\rho,B)
\stackrel{\mu_1}\longleftrightarrow(\{x_5,x_4\},\rho,B')
\stackrel{\mu_2}\longleftrightarrow (\{x_5,x_6\},\rho,B)
\\ \phantom{\cdots}
 {} \stackrel{\mu_1}\longleftrightarrow(\{x_7,x_6\},\rho,B')
\stackrel{\mu_2}\longleftrightarrow (\{x_7,x_8\},\rho,B)\stackrel{\mu_1}\longleftrightarrow\cdots,
\end{gather*}
where the cluster variables $x_k$ for $k\in \mathbb{Z}$ satisfy the exchange relations:
\begin{gather*}
 x_{k-1}x_{k+1}=
 \begin{cases}
 1+x_k &\text{if}\quad k \ \text{is odd},
 \\
 1+hx_k+x_{k}^{2}&\text{if}\quad k\ \text{is even}.
 \end{cases}
\end{gather*}
By a direct calculation, we have that
\begin{gather*}
 x_3=x^{-1}_1+hx^{-1}_1x^{}_2+x^{-1}_1x^{2}_2,\\
 x_4= x^{-1}_2+x^{-1}_1x^{-1}_2+hx^{-1}_1+x^{-1}_1x^{}_2,\\
 x_5=x^{-1}_1+hx^{-1}_1x^{-1}_2+hx^{-1}_2+x^{-1}_1x^{-2}_2+2x^{-2}_2+x^{}_1x^{-2}_2,\\
 x_6=x_0=x^{-1}_2+x^{}_1x^{-1}_2, \qquad
 x_7=x_1, \qquad x_8=x_2.
\end{gather*}
It is now clear that the sequence of~cluster variables $\{x_k\}_{k\in\mathbb{Z}}$ is $6$-periodic. Thus we have only six distinct cluster variables.
\end{Example}

\sloppy\begin{Definition}\label{def of gca}
For an~initial generalized seed $\big(\widetilde{\mathbf{x}}, \rho, \widetilde{B}\big)$, the generalized cluster algebra $\smash{\mathcal{A}\big(\widetilde{\mathbf{x}}, \rho, \widetilde{B}\big)}$ is the $\mathop{\mathbb{ZP}}$-subalgebra of~$\mathbb{Q}(x_1,\ldots,x_m)$ generated by all cluster variables from all gene\-ralized seeds which are mutation-equivalent to $\smash{\big(\widetilde{\mathbf{x}},\rho,\widetilde{B}\big)}$. The integer $n$ is the rank of~$\smash{\mathcal{A}\big(\widetilde{\mathbf{x}}, \rho, \widetilde{B}\big)}$.
\end{Definition}

For each $i\in{\rm\mathbf{ex}}$, we define
\begin{gather*}
 P_i:=x_{i}x'_{i}=\sum\limits_{r=0}^{d_i}\rho_{i,r}\prod\limits_{j=1}^{m} x_{j}^{r[\beta_{ji}]_+ +(d_i-r)[-\beta_{ji}]_+}.
\end{gather*}
It follows that $P_i\in\mathop{\mathbb{ZP}}[x_1,\ldots,x_{i-1},x_{i+1},\ldots,x_n]$.

\begin{Definition}
The generalized seed $\big(\widetilde{\mathbf{x}},\rho,\widetilde{B}\big)$ is called coprime if $P_i$ and $P_j$ are coprime for~any two different $i$, $j\in{\rm\mathbf{ex}}$.
\end{Definition}

Let~$\big(\widetilde{\mathbf{x}},\rho,\widetilde{B}\big)$ be a generalized seed. The directed graph $\Gamma\big(\widetilde{\mathbf{x}},\rho,\widetilde{B}\big)$ is defined as follows:
\begin{itemize}\itemsep=0pt
 \item[(1)] its vertices consist of~all $i\in{\rm\mathbf{ex}}$;
 \item[(2)] a pair $(i,j)$ is a directed edge of~$\Gamma\big(\widetilde{\mathbf{x}},\rho,\widetilde{B}\big)$ if and only if $b_{ij}>0$.
\end{itemize}

\begin{Definition}
The generalized seed $\big(\widetilde{\mathbf{x}},\rho,\widetilde{B}\big)$ is called acyclic if $\Gamma\big(\widetilde{\mathbf{x}},\rho,\widetilde{B}\big)$ does not contain any oriented cycle.
A~gene\-ra\-li\-zed cluster algebra $\mathcal{A}$ is called acyclic if it has an~acyclic gene\-ra\-li\-zed~seed.
\end{Definition}

The following definition is a natural generalization of~\cite[Definition~1.15]{bfz}.
\begin{Definition}
Let~$\big(\widetilde{\mathbf{x}},\rho,\widetilde{B}\big)$ be a generalized seed. A~standard monomial in $\{x_i,x'_i\,|\, i\in{\rm\mathbf{ex}}\}$ is a monomial that does not have any factor of~the form $x_ix'_i$ for any $i\in{\rm\mathbf{ex}}$.
\end{Definition}

In order to define the upper bounds and lower bounds, we write ${\rm\mathbf{ex}}=\{i_1,\ldots,i_n\}$.
\begin{Definition}
For a generalized seed $\big(\widetilde{\mathbf{x}},\rho,\widetilde{B}\big)$, the upper bound is defined by
\begin{gather*}
\mathcal{U}(\widetilde{\mathbf{x}}, \rho, \widetilde{B}):= \mathop{\mathbb{ZP}}\!\big[x_{i_1}^{\pm1},\ldots,x_{i_n}^{\pm1}\big]\cap
\bigcap\limits_{k=1}^{n}\mathop{\mathbb{ZP}}\!\big[x_{i_1}^{\pm1},\ldots,x_{i_{k-1}}^{\pm1},(x'_{i_k})^{\pm1},x_{i_{k+1}}^{\pm1}, \ldots,x_{i_n}^{\pm1}\big]
\end{gather*}
and the lower bound by
\begin{gather*}
\mathcal{L}\big(\widetilde{\mathbf{x}},\rho,\widetilde{B}\big):= \mathop{\mathbb{ZP}}[x_{i_1},x'_{i_1},\ldots,x_{i_n},x'_{i_n}].
\end{gather*}
\end{Definition}

Note that
\begin{gather*}
 \mathcal{L}\big(\widetilde{\mathbf{x}},\rho,\widetilde{B}\big)\subseteq \mathcal{A}\big(\widetilde{\mathbf{x}},\rho,\widetilde{B}\big) \subseteq \mathcal{U}\big(\widetilde{\mathbf{x}},\rho,\widetilde{B}\big).
\end{gather*}

\begin{Theorem}[{\cite[Theorem 4.1]{gsv18}}]
Let~$i\in{\rm\mathbf{ex}}$. If the generalized seeds $\big(\widetilde{\mathbf{x}},\rho,\widetilde{B}\big)$ and $\mu_i\big(\widetilde{\mathbf{x}},\rho,\widetilde{B}\big)$ are coprime, then we have
\begin{gather*}
\mathcal{U}\big(\widetilde{\mathbf{x}},\rho,\widetilde{B}\big) =\mathcal{U}\big(\mu_i\big(\widetilde{\mathbf{x}},\rho,\widetilde{B}\big)\big).
\end{gather*}
\end{Theorem}

The following definition is a generalization of~\cite[Section~3.1]{muller14}.
\begin{Definition}\label{def of freezing}
Let~$\big(\widetilde{\mathbf{x}},\rho,\widetilde{B}\big)$ be a generalized seed and $i\in{\rm\mathbf{ex}}$. A~new generalized seed $\big(\widetilde{\mathbf{x}}^{\dag},\rho^{\dag},\widetilde{B}^{\dag}\big)$ is defined as follows:
\begin{itemize}\itemsep=0pt
 \item[(1)] let $\mathop{\mathbb{ZP}}\nolimits^{\dag}=\mathop{\mathbb{ZP}}\big[x^{\pm1}_i\big]$, $\mathbf{x}^{\dag}=\big\{ x_k\,|\, k\in{\rm\mathbf{ex}} -\{i\} \big\}$ and $\widetilde{\mathbf{x}}^{\dag}- \mathbf{x}^{\dag}=(\mathbf{\widetilde{x}}- \mathbf{x})\cup\{x_i\}$;
 \item[(2)] the matrix $\widetilde{B}^{\dag}$ is obtained from $\widetilde{B}$ by deleting the $i$th column, the principal part $B^{\dag}$ is the submatrix of~$\widetilde{B}^{\dag}$ with rows labeled by ${\rm\mathbf{ex}}-\{i\}$;
 \item[(3)] let $\rho^{\dag}=\rho-\{\rho_i\}$.
\end{itemize}
The generalized seed\,$\big(\widetilde{\mathbf{x}}^{\dag},\rho^{\dag},\widetilde{B}^{\dag}\big)$ is called the freezing of\,$\big(\widetilde{\mathbf{x}},\rho,\widetilde{B}\big)$\,at~$x_i$.
The freezing of\,$\mathcal{A}\big(\widetilde{\mathbf{x}},\rho,\widetilde{B}\big)$ at $x_i$ is defined to be the generalized cluster algebra $\smash{\mathcal{A}\big(\widetilde{\mathbf{x}}^{\dag},\rho^{\dag},\widetilde{B}^{\dag}\big)}$, which is the $\mathop{\mathbb{ZP}}\nolimits^{\dag}$-subalgebra of~$\mathbb{Q}(x_1,\ldots,x_{m})$ generated by all cluster variables from the generalized seeds which are mutation-equivalent to $\smash{\big(\widetilde{\mathbf{x}}^{\dag},\rho^{\dag},\widetilde{B}^{\dag}\big)}$.

\end{Definition}

\begin{Example}
Let~$(\mathbf{x},\rho,B)$ be the generalized seed from Example~\ref{example1}. Let~$\mathbf{x}^{\dag}=\{x_2\}$, $\rho^\dag=\{\rho_2\}$ and $B^\dag=(0)$. Then the generalized seed $\big(\mathbf{x}^{\dag},\rho^\dag,B^\dag\big)$ is the freezing of~$(\mathbf{x},\rho,B)$ at $x_1$. It follows that the generalized cluster algebra
$\mathcal{A}\big(\mathbf{x}^{\dag},\rho^\dag,B^\dag\big)=\mathbb{Z}\big[x^{\pm1}_1,x_2, \frac{x_1+1}{x_2}\big]$ is the freezing of~$\mathcal{A}(\mathbf{x},\rho,B)$ at $x_1$.
Similarly, the generalized seed $\big(\{x_1\},\{\rho_1\},(0)\big)$ is the freezing of~$(\mathbf{x},\rho,B)$ at $x_2$ and the freezing of~$\mathcal{A}(\mathbf{x},\rho,B)$ at $x_2$ is $\mathcal{A}\big(\{x_1\},\{\rho_1\},(0)\big)=\mathbb{Z}\big[x_1,x^{\pm1}_2,\frac{1+hx_2+x^{2}_2}{x_1}\big]$.
\end{Example}

The freezing at $x_i$ is compatible with the mutation in direction $j$ for $i\neq j$, therefore we have the following result.

\begin{Lemma}
Let~$i$, $j\in{\rm\mathbf{ex}}$ be distinct and assume $\mu_j\big(\widetilde{\mathbf{x}},\rho,\widetilde{B}\big)
=\big(\widetilde{\mathbf{y}},\varepsilon,\widetilde{A}\big)$. If $\big(\widetilde{\mathbf{x}}^{\dag},\rho^{\dag},\widetilde{B}^{\dag}\big)$ is the freezing of~$\smash{\big(\widetilde{\mathbf{x}},\rho,\widetilde{B}\big)}$ at $x_i$ and $\smash{\big(\widetilde{\mathbf{y}}^{\dag},\varepsilon^{\dag},\widetilde{A}^{\dag}\big)}$ the freezing of~$\smash{\big(\widetilde{\mathbf{y}},\varepsilon,\widetilde{A}\big)}$ at $y_i$, then we have that $\mu_j\smash{\big(\widetilde{\mathbf{x}}^{\dag},\rho^{\dag},\widetilde{B}^{\dag}\big)
=\big(\widetilde{\mathbf{y}}^{\dag},\varepsilon^{\dag},\widetilde{A}^{\dag}\big)}$.
\end{Lemma}

\begin{proof}
The proof is straightforward, so we omit the details.
\end{proof}

The freezing of~$\smash{\big(\widetilde{\mathbf{x}},\rho,\widetilde{B}\big)}$ at $\{x_{j_1},\ldots,x_{j_k}\} \subsetneq \mathbf{x}$ is the generalized seed obtained from $\smash{\big(\widetilde{\mathbf{x}},\rho,\widetilde{B}\big)}$ by iterated freezing at each cluster variable in $\{x_{j_1},\ldots,x_{j_k}\}$ in any order. Note that the freezing at $\{x_{j_1},\ldots,x_{j_k}\}$ is compatible with the mutation in direction $l$ for $l\notin\{j_1,\ldots,j_k\}$. It will cause no confusion if we still use $\smash{\big(\widetilde{\mathbf{x}}^{\dag},\rho^{\dag},\widetilde{B}^{\dag}\big)}$ to denote the resulting generalized seed. In the same manner, the generalized cluster algebra $\mathcal{A}\smash{\big(\widetilde{\mathbf{x}}^{\dag},\rho^{\dag},\widetilde{B}^{\dag}\big)}$ is called the freezing of~$\mathcal{A}\smash{\big(\widetilde{\mathbf{x}},\rho,\widetilde{B}\big)}$ at~$\{x_{j_1},\ldots,x_{j_k}\}$. The rank of~$\mathcal{A}\smash{\big(\widetilde{\mathbf{x}}^{\dag},\rho^{\dag},\widetilde{B}^{\dag}\big)}$ is $n-k$.

\section{Lower bounds and upper bounds}
We follow the arguments in \cite{bfz} to prove that lower bounds and upper bounds coincide under the assumptions of~acyclicity and coprimality.
We once defined the initial $\mathbf{ex}$ to be any subset of~$[1,m]$ with $|\mathbf{ex}|=n$ in order to easily describe the freezing in Definition~\ref{def of freezing}. For the sake of~convenience, for the remainder of~the paper, we will take $\mathbf{ex}=[1,n]$ after renumbering the indices of~the original $\mathbf{ex}$.
Let~$\smash{\big(\widetilde{\mathbf{x}},\rho,\widetilde{B}\big)}$ be a generalized seed and $B=(b_{ij})$ the corresponding principal part. Recall that $\smash{\big(\widetilde{\mathbf{x}},\rho,\widetilde{B}\big)}$ is acyclic if and only if there exists a permutation $\sigma\in S_n$ such that $b_{\sigma(l),\sigma(k)}\geq 0$ for $1\leq k<l\leq n$. Hence we can assume that $b_{lk}\geq0$ for $1\leq k<l\leq n$ if the generalized seed $\smash{\big(\widetilde{\mathbf{x}},\rho,\widetilde{B}\big)}$ is acyclic.

\begin{Theorem}\label{thmlinind}
If the generalized seed $\big(\widetilde{\mathbf{x}},\rho,\widetilde{B}\big)$ is acyclic, then the standard monomials in $x_1, \allowbreak x'_1, \ldots, x_n, x'_n$ are $\mathop{\mathbb{ZP}}$-linearly independent in
$\smash{\mathcal{L}\big(\widetilde{\mathbf{x}},\rho,\widetilde{B}\big)}$.
\end{Theorem}
\begin{proof}
By using the same technique as in \cite{bfz}, we can prove the statement and we write the proof down here for readers' convenience.

For any $\mathbf{a}=(a_1,\ldots,a_n)\in\mathbb{Z}^{n}$, we denote by
\begin{gather*}
\mathbf{x}^{(\mathbf{a})}=x^{(a_1)}_1 x^{(a_2)}_2\cdots x^{(a_n)}_n
\end{gather*}
the standard monomial in $\{x_1,x'_1,\ldots,x_n,x'_n\}$,
where $x^{(a_i)}_i= x^{a_i}_i$ if $a_i\geq0$, and $x^{(a_i)}_i= (x'_i)^{-a_i}$ if~$a_i<0$, and we denote the Laurent monomial $x^{a_1}_1x^{a_2}_2\cdots x^{a_n}_n$ by $\mathbf{x}^{\mathbf{a}}$.

Let~``$\prec$" denote the lexicographic order on $\mathbb{Z}^{n}$ which induces
the lexicographic order on the Laurent monomials as
\begin{gather*}
\mathbf{x}^{\mathbf{a}}\prec\mathbf{x}^{\mathbf{a}'}
\qquad\text{if}\qquad
\mathbf{a}\prec\mathbf{a}'.
\end{gather*}

Note that $x^{(-1)}_i=x'_i=x^{-1}_iP_i$ for $i\in [1,n]$. Using the assumption that $b_{lk}\geq0$ for $l>k$, it follows that the lexicographically first monomial which appears in $x^{(-1)}_i$ is $x^{-1}_i\prod\limits_{j=i+1}^{n}x^{b_{ji}}_j\!\! \prod\limits_{k=n+1}^{m}x^{[b_{ki}]_+}_k$. We then conclude that the lexicographically first monomial that appears in $\mathbf{x}^{(\mathbf{a}')}$
 is preceded by~the one in $\mathbf{x}^{(\mathbf{a})}$
 if $\mathbf{a}\prec \mathbf{a}'$. Therefore the standard monomials in $\{x_1,x'_1,\ldots,x_n,x'_n\}$ are $\mathop{\mathbb{ZP}}$-linearly independent.
\end{proof}

\begin{Remark}It is to be expected that the converse of~the above theorem, i.e., ``For a gene\-ra\-li\-zed seed $\smash{\big(\widetilde{\mathbf{x}},\rho,\widetilde{B}\big)}$, if the standard monomials in $x_1, x'_1, \ldots, x_n, x'_n$ are $\mathop{\mathbb{ZP}}$-linearly independent in~$\mathcal{L}\smash{\big(\widetilde{\mathbf{x}},\rho,\widetilde{B}\big)}$, then the generalized seed $\smash{\big(\widetilde{\mathbf{x}},\rho,\widetilde{B}\big)}$ is acyclic" is also true as it was proved in~\mbox{\cite[Proposition~5.1]{bfz}} for cluster algebras. But we have not been able to prove this.
\end{Remark}

The following two lemmas are generalizations of~\cite[Lemmas~4.1 and~4.2]{bfz}. The proofs are omitted as they are similar to proofs in~\cite{bfz}.
\begin{Lemma}\label{lemup}
We have that
\begin{gather}\label{equxxprime}
\mathcal{U}\big(\widetilde{\mathbf{x}},\rho,\widetilde{B}\big)= \bigcap\limits_{i=1}^{n} \mathop{\mathbb{ZP}}\big[x_{1}^{\pm1},\ldots,x_{i-1}^{\pm1},x_{i},x_{i}',x_{i+1}^{\pm1},\ldots,x_{n}^{\pm1}\big].
\end{gather}
\end{Lemma}

\begin{Lemma}\label{lemdivisible}
Given $y\in\mathbb{Q}(x_1,\ldots,x_m)$, the element $y\in\mathop{\mathbb{ZP}}\big[x_{1},x_{1}',x_{2}^{\pm1},\ldots,x_{n}^{\pm1}\big]$ if and only if $y=\sum\limits_{k=a}^{b}c_k x_{1}^{k}$ $(a\leq b)$ satisfies one of~the following conditions:
\begin{itemize}\itemsep=0pt
 \item[{\rm (1)}] if $a\geq0$, we have $c_k\in\mathop{\mathbb{ZP}}\!\big[x_{2}^{\pm1},\ldots,x_{n}^{\pm1}\big];$
 \item[{\rm (2)}] if $a<0$, we have that $c_k\in\mathop{\mathbb{ZP}}\!\big[x_{2}^{\pm1},\ldots,x_{n}^{\pm1}\big]$ and $c_kP_{1}^{k}\in\mathop{\mathbb{ZP}}\!\big[x_{2}^{\pm1},\ldots,x_{n}^{\pm1}\big]$ for $k<0$ $\big($namely, $c_k$ is divisible by $P_{1}^{|k|}$ in $\smash{\mathop{\mathbb{ZP}}\!\big[x_{2}^{\pm1},\ldots,x_{n}^{\pm1}\big]\big)}$.
\end{itemize}
\end{Lemma}

Let~$\mathop{\mathbb{ZP}}\nolimits^{\rm st}[x_{2},x'_{2},\ldots,x_n,x'_{n}] \subseteq $ $\mathop{\mathbb{ZP}} [x_{2},x'_{2},\ldots,x_n,x'_{n}]$ denote the $\mathop{\mathbb{ZP}}$-linear space spanned by~the standard monomials in $\{x_{2},x'_{2},\ldots,x_n,x'_{n}\}$. Similarly, we denote by
\begin{gather*}
\mathop{\mathbb{ZP}}\nolimits^{\rm st}[x_1,x_{2},x'_{2},\ldots,x_n,x'_{n}]
\end{gather*}
the $\mathop{\mathbb{ZP}}[x_1]$-linear space spanned by the standard monomials in $\{x_{2},x'_{2},\ldots,x_n,x'_{n}\}$.
Using the exchange relations
\begin{gather*}
x_ix'_i=P_i\in\mathop{\mathbb{ZP}}[x_1,\ldots,x_{i-1},x_{i+1},\ldots,x_n]
\end{gather*}
repeatedly, it follows that any element in $\mathop{\mathbb{ZP}}[x_{2},x'_{2},\ldots,x_n,x'_{n}]$ is a $\mathop{\mathbb{ZP}}[x_1]$-linear combination of~standard monomials in $x_{2}, x'_{2}, \ldots, x_n, x'_{n}$. Thus, we have{\samepage
\begin{gather*}
\mathop{\mathbb{ZP}}[x_{2},x'_{2},\ldots,x_n,x'_{n}] \subseteq \mathop{\mathbb{ZP}}\nolimits^{\rm st}[x_1,x_{2},x'_{2},\ldots,x_n,x'_{n}].
\end{gather*}
Note that, this relation is the analogue of~\cite[equation~(6.1)]{bfz}.}

Define
\begin{gather*}
f_1\colon\ \mathop{\mathbb{ZP}}\big[x_{2},x'_{2},\ldots,x_n,x'_{n}]\rightarrow \mathop{\mathbb{ZP}}[x_1,x^{\pm1}_2,\ldots,x^{\pm1}_{n}\big]\qquad \text{by}\quad
x_i\mapsto x_i \quad\text{and}\quad x'_{i}\mapsto x_{i}^{-1}P_i
\end{gather*}
and
\begin{gather*}
f_2\colon\ \mathop{\mathbb{ZP}}\big[x_{1},x^{\pm1}_{2},\ldots,x^{\pm1}_{n}\big]\rightarrow \mathop{\mathbb{ZP}}\big[x^{\pm1}_2,\ldots,x^{\pm1}_{n}\big]\qquad \text{by}\quad
x_1\mapsto 0 \quad\text{and} \quad x_{i}^{\pm1}\mapsto x_{i}^{\pm1}.
\end{gather*}
Both are algebra homomorphisms, and consequently,
\begin{gather*}
f:=f_2\circ f_1\colon\ \mathop{\mathbb{ZP}}[x_{2},x'_{2},\ldots,x_n,x'_{n}]\rightarrow \mathop{\mathbb{ZP}}\big[x^{\pm1}_2,\ldots,x^{\pm1}_{n}\big]
\end{gather*}
is also an~algebra homomorphism.

Given $y\in\mathop{\mathbb{ZP}}\!\big[x^{\pm1}_1,\ldots, x^{\pm1}_{n}\big]$, as in \cite[Definition~6.3]{bfz}, we also define the leading term $\mathop{\rm LT}(y)$ of~$y$ with respect to $x_1$ to be the sum of~Laurent monomials with the smallest power of~$x_1$, which are obtained from the Laurent expression of~$y$ with non-zero coefficient.

The following results parallel to \cite[Lemmas 6.2,~6.4 and~6.5]{bfz} can be obtained similarly.
\begin{Lemma}\label{lemdirectsum}
We have that
\begin{itemize}\itemsep=0pt
 \item[{\rm (1)}] $\mathop{\mathbb{ZP}}[x_{2},x'_{2},\ldots,x_n,x'_{n}]={\rm Ker}(f)\oplus \mathop{\mathbb{ZP}}\nolimits^{\rm st}[x_{2},x'_{2},\ldots,x_n,x'_{n}];$
 \item[{\rm (2)}] if $y\in\mathop{\mathbb{ZP}}\!\big[x^{\pm1}_1,x_2,x'_2,\ldots,x_n,x'_{n}\big]$ and $y=\sum\limits_{k=a}^{b}c_kx_{1}^{k}$ such that we have $c_a\neq0$ and $c_k\in \mathop{\mathbb{ZP}}\nolimits^{\rm{st}}[x_2,x'_2,\ldots,x_n,x'_{n}]$, then ${\rm LT}(y)=f(c_a)x_{1}^{a};$
\item[{\rm (3)}] $\mathop{\mathbb{ZP}}\!\big[x_{1}^{\pm1},x_{2},x_{2}',\ldots,x_{n},x_{n}'\big] \cap \mathop{\mathbb{ZP}}\!\big[x_{1},x_{2}^{\pm1},\ldots,x_{n}^{\pm1}\big] =\mathop{\mathbb{ZP}}[x_{1},x_{2},x_{2}',\ldots,x_{n},x_{n}']$.
\end{itemize}
\end{Lemma}

\begin{Lemma}\label{lemim}
We have that
\begin{gather*}
\mathop{\rm Im} (f)=\mathop{\mathbb{ZP}}\big[x_2,\widehat{x}_{2},\ldots,x_n,\widehat{x}_{n}\big],
\end{gather*}
where $\widehat{x}_{i}:= x_{i}'$ if $b_{1i}=0$ and $\widehat{x}_{i}:= x_{i}^{-1}$ otherwise.
\end{Lemma}

\begin{proof}
By suitable and non-trivial modifications of~the proof of~\cite[Lemma~6.6]{bfz}, we can prove the statement. By a direct calculation, for $j\in[2,n]$ one can show that
\begin{gather*}
 f(x_{j}') = \begin{cases}
 x_{j}' &\text{if}\quad b_{1j}=0,
 \\
 x_{j}^{-1}\prod\limits_{i=j+1}^{n}x_{i}^{b_{ij}}\prod\limits_{i=n+1}^{m}x_{i}^{[b_{ij}]_+}
 &\text{if} \quad b_{1j}\neq 0.
\end{cases}
\end{gather*}

The inclusion $\mathop{\rm Im}(f)\subseteq\mathop{\mathbb{ZP}}\big[x_2,\widehat{x}_{2},\ldots,x_n,\widehat{x}_{n}\big]$ is immediate. For each $j\in[2,n]$, we set
$M_j=x_{j}^{-1}\prod\limits_{i=j+1}^{n}x_{i}^{b_{ij}}.$ Let~$J=\{j\in[2,n]\,|\, b_{1j}=0\}$. In order to prove the converse inclusion, it~suf\-fi\-ces to show that $x_{j}^{-1}\in\mathop{\rm Im}(f)$ for each $j\in[2,n]-J$\!. For $l_j\in \mathbb{Z}$, we have
$\prod\limits_{j=2}^{n}(M_{j})^{l_j}=\prod\limits_{j=2}^{n}
\Big(x_{j}^{-1}\prod\limits_{i=j+1}^{n}x_{i}^{b_{ij}}\Big)^{l_j}
=\prod\limits_{j=2}^{n}x_{j}^{z_j}$ with $z_j:= {-l_j+\sum\limits_{t=2}^{j-1}b_{jt}l_{t}}$.
We define the multiplicative monoid
\begin{gather*}
\mathfrak{M}:=\Bigg\{\prod\limits_{k=2}^{n}M_{k}^{l_k}=\prod\limits_{j=2}^{n}x_{j}^{z_j} \,|\, l_k\geq0 \text{ for } k\in[2,n],\ z_j\geq0\text{ for } j \in J\Bigg\}.
\end{gather*}
Then we have that $\prod\limits_{k=2}^{n}(M_{k})^{l_k}\in\mathfrak{M}$ if and only if
\begin{gather}\label{equm}
l_k\geq0 \quad\text{for}\quad k\in[2,n]\qquad\text{and}\qquad
l_j\leq \sum\limits_{t=2}^{j-1}b_{jt}l_{t}\quad\text{for}\quad j\in J.
\end{gather}

For any $j\in [2,n]-J$, we obtain that $x_{j}^{-1}\in \mathfrak{M}$ by (\ref{equm}). Hence it is sufficient to prove that $\mathfrak{M}\subseteq \mathop{\rm Im}(f)$. Let~$M:=\prod\limits_{k=2}^{n}M_{k}^{l_k}\in\mathfrak{M}$. We will prove $M\in\mathop{\rm Im}(f)$ by induction on~\mbox{$\text{deg}(M):=\sum\limits_{k=2}^{n}l_{k}$}. When $\text{deg}(M)=0$, we have $M=1\in\mathop{\rm Im}(f)$. Suppose that $\text{deg}(M)>0$. Now we assume that each $M'\in\mathop{\rm Im}(f)$ if $M'\in\mathfrak{M}$ whose degree is no more than $\text{deg}(M)-1$.

Let~$j\in[2,n]$ be the largest integer such that $l_{j}>0$. Since $l_{j}-1 <\sum\limits_{t=2}^{j-1}b_{jt}l_{t}$ if $l_{j} \leq\sum\limits_{t=2}^{j-1}b_{jt}l_{t}$ and $l_j-1\geq0$, we have that $M/M_j\in\mathfrak{M}$. By the induction hypothesis, we obtain that $M/M_j\in\mathop{\rm Im}(f)$. If $j\notin J$, then $f(x_{j}')=M_j\prod\limits_{i=n+1}^{m}x_{i}^{[b_{ij}]_+}$, which implies $M_j\in \mathop{\rm Im}(f)$. It~fol\-lows that $M=(M/M_j) M_j\in\mathop{\rm Im}(f)$. Now assume that $j\in J$. Using the fact that $b_{lk}\geq0$ for~$1\leq k< l\leq n$, it follows that
\begin{gather*}
 x_{j}'= x_{j}^{-1}\sum\limits_{r=0}^{d_j}\rho_{j,r} \prod\limits_{i=2}^{j-1}x_{i}^{(r-d_j)\beta_{ij}} \prod\limits_{i=j+1}^{n}x_{i}^{r\beta_{ij} } \prod\limits_{i=n+1}^{m}x_{i}^{r[\beta_{ij}]_{+} +(d_j-r)[-\beta_{ij}]_+}
 \\ \hphantom{x_{j}'}
{}= M_j\prod\limits_{i=n+1}^{m}x_{i}^{[b_{ij}]_{+}} +x_{j}^{-1}\sum\limits_{r=0}^{d_j-1}\rho_{j,r} \prod\limits_{i=2}^{j-1}x_{i}^{(r-d_j)\beta_{ij}} \prod\limits_{i=j+1}^{n}x_{i}^{r\beta_{ij} } \prod\limits_{i=n+1}^{m}x_{i}^{r[\beta_{ij}]_{+} +(d_j-r)[-\beta_{ij}]_+}.
\end{gather*}
By multiplying both sides of~this equation by $M/M_j$, we have that
\begin{gather*}
(M/M_j)\, x_{j}'= M\prod\limits_{i=n+1}^{m}x_{i}^{[b_{ij}]_{+}} + \sum\limits_{r=0}^{d_j-1}\rho_{j,r} M\prod\limits_{i=2}^{n}x_{i}^{(r-d_j)\beta_{ij}} \prod\limits_{i=n+1}^{m}x_{i}^{r[\beta_{ij}]_{+} +(d_j-r)[-\beta_{ij}]_+}.
\end{gather*}
Note that $(M/M_j) x_{j}'\in\mathop{\rm Im}(f)$. If $M\prod\limits_{i=2}^{n}x_{i}^{(r-d_j)\beta_{ij}} \in \mathop{\rm Im}(f)$ for all $r\in[0,d_j-1]$, then $M\in\mathop{\rm Im}(f)$.
Thus we only need to prove that $M\prod\limits_{i=2}^{n}x_{i}^{(r-d_j)\beta_{ij}} \in \mathop{\rm Im}(f)$ for $r\in[0,d_j-1]$. Since $\prod\limits_{i=j+1}^{n}x_{i}^{r\beta_{ij}}\in \mathop{\rm Im}(f)$, it is sufficient to prove that
\begin{gather*}
M\prod\limits_{i=j+1}^{n}x_{i}^{-b_{ij}} \prod\limits_{i=2}^{j-1}x_{i}^{(r-d_j)\beta_{ij}} =(M/M_j)\Bigg(x_{j}^{-1} \prod\limits_{i=2}^{j-1}x_{i}^{(r-d_j)\beta_{ij}}\Bigg) \in \mathop{\rm Im}(f).
\end{gather*}
For simplicity, we denote $M_{j}^{-}:=x_{j}^{-1}\prod\limits_{i=2}^{j-1}x_{i}^{-b_{ij}}$. Then
\begin{gather*}
(M/M_j)\Bigg(x_{j}^{-1} \prod\limits_{i=2}^{j-1}x_{i}^{(r-d_j)\beta_{ij}}\Bigg)
=(M/M_j)M_{j}^{-}\prod\limits_{i=2}^{j-1}x_{i}^{r\beta_{ij}}.
\end{gather*}

By applying the facts that $l_j>0$ for any $j\in J$ and (\ref{equm}), we know that there exists the smallest integer $h\in[2,j-1]$ such that $b_{jh}l_h>0$ which implies $b_{jh}>0$, $l_h>0$ and $b_{jk}l_k=0$ for~$k\in[2,h-1]$. By (\ref{equm}), we know that $h\in[2,n]-J$.
Let~$M':=\prod\limits_{k=2}^{n}M_{k}^{l'_k}$ be a monomial such that
\begin{gather}\label{equalprime}
l_{k}'=
\begin{cases}
0 &\text{for}\quad 2\leq k\leq h-1,
\\
1 &\text{for}\quad k=h,
\\
\min\bigg\{l_k,\sum\limits_{t=2}^{k-1}b_{kt}l_{t}'\bigg\}&\text{for}\quad h+1\leq k\leq n.
\end{cases}
\end{gather}

Note that $l'_k\geq0$ since $b_{kt}\geq0$ for $2\leq t< k\leq n$. By $h\in[2,n]-J$ and (\ref{equalprime}), we know that $l'_k\leq \sum\limits_{t=2}^{k-1}b_{kt}l_{t}'$ for $k\in J$. Hence $M'\in\mathfrak{M}$. It is clear that $\text{deg}(M')\geq1$.

In order to prove $M/M'\in\mathfrak{M}$, we need to verify that
\begin{gather}{\label{inequl}}
l'_k\leq l_k\qquad\text{for}\quad k\in[2,n]
\end{gather}
and
\begin{gather}\label{inequmprime}
{}-l'_k +\sum\limits_{t=2}^{k-1}b_{kt}l_{t}' \leq -l_k +\sum\limits_{t=2}^{k-1}b_{kt}l_{t}
\qquad\text{for}\quad k\in J.
\end{gather}
By (\ref{equalprime}) and the choice of~$h$, we have that the inequalities (\ref{inequl}) hold.

We show that the inequalities (\ref{inequmprime}) hold as follows:
\begin{itemize}\itemsep=0pt
 \item[(a)] it is clear that $-l'_k +\sum\limits_{t=2}^{k-1}b_{kt}l_{t}'=0\leq -l_k +\sum\limits_{t=2}^{k-1}b_{kt}l_{t}$ for $k\in J\cap [2,h-1]$;
 \item[(b)] when $k=h$, recall that $h\notin J$;
 \item[(c)] for $k\in[h+1,n]$, we have $l_{k}'=l_k$ or $l_{k}'=\sum\limits_{t=2}^{k-1}b_{kt}l_{t}'$. If $l_{k}'=l_k$, then the inequalities (\ref{inequmprime}) hold since $l'_t\leq l_{t}$ for $2\leq t\leq k-1$. If $l_{k}'=\sum\limits_{t=2}^{k-1}b_{kt}l_{t}'$, then $-l'_k +\sum\limits_{t=2}^{k-1}b_{kt}l_{t}' =0 \leq -l_k +\sum\limits_{t=2}^{k-1}b_{kt}l_{t}$ for $k\in J\cap[h+1,n]$. Thus we obtain the inequalities (\ref{inequmprime}).
\end{itemize}

The inequalities (\ref{inequl}) and (\ref{inequmprime}) imply that $M/M'\in\mathfrak{M}$. By (\ref{inequl}), we have that $\text{deg}(M/M')<\text{deg}(M)$. The induction hypothesis implies that $M/M'\in\mathop{\rm Im}(f)$.

Note that $(M/M_j)M_{j}^{-}\!\prod\limits_{k=2}^{j-1}\!x_{k}^{r\beta_{kj}}=(M/M')(M'M_{j}^{-}/M_j) \prod\limits_{k=2}^{j-1}\!x_{k}^{r\beta_{kj}}.$
Since we know that \mbox{$M/M'\in\mathop{\rm Im}(f)$}, it suffices to show that$(M'M_{j}^{-}/M_j) \!\prod\limits_{k=2}^{j-1}x_{k}^{r\beta_{kj}}\in\mathop{\rm Im}(f)$.

Let~$\prod\limits_{k=2}^{n}x_{k}^{c_k}=(M'M_{j}^{-}/M_j) \prod\limits_{t=2}^{j-1}x_{t}^{r\beta_{tj}}$ with $ c_k \in \mathbb{Z}$ for all $k \in [2,n]$. To prove that $\prod\limits_{k=2}^{n}x_{k}^{c_k}\in\mathop{\rm Im}(f)$, we only need to show that $c_k\geq0$ for all $k\in[2,n]$. Note that $0\leq r\leq d_{j}-1$.

By calculating the powers of~$x_k$ in $M'$, $M_{j}^{-}$, $M_j$ and $\prod\limits_{t=2}^{j-1}x_{t}^{r\beta_{tj}}$, respectively, it follows that:
\begin{itemize}\itemsep=0pt
 \item[(1)] when $k\in[2,h-1]$, we have $\beta_{kj}\leq0$ and
 \begin{gather*}
 c_k=0+(-b_{kj})-0+r\beta_{kj}=(r-d_{j})\beta_{kj}\geq 0;
 \end{gather*}
 \item[(2)] when $k=h$, since $b_{jh}>0$, we obtain $\beta_{hj}<0$. It follows that
 \begin{gather*}
 c_h=-l_{h}'+(-b_{hj})-0+r\beta_{hj}=-1+(r-d_{j})\beta_{hj}\geq -\beta_{hj}-1\geq0;
 \end{gather*}
 \item[(3)] when $k\in[h+1,j-1]$, then $\beta_{kj}\leq0$ and by (\ref{equalprime}) we have that
 \begin{gather*}
 c_k=\Bigg(-l'_k +\sum\limits_{t=2}^{k-1}b_{kt}l_{t}'\Bigg)+0-b_{kj}+r\beta_{kj} \geq(r-d_{j})\beta_{kj} \geq0;
 \end{gather*}
 \item[(4)] when $k=j$, by (\ref{equalprime}) we obtain that
 \begin{gather*}
 c_j=\Bigg(-l'_j +\sum\limits_{t=2}^{j-1}b_{jt}l_{t}'\Bigg)+(-1)-(-1)+0= -l'_j +\sum\limits_{t=2}^{j-1}b_{jt}l_{t}'\geq0;
 \end{gather*}
 \item[(5)] when $k\in[j+1,n]$, note that $l_k=0$. But $0\leq l'_k\leq l_k$ which implies $l'_k=0$. Since $\sum\limits_{t=2}^{j-1}b_{jt}l_{t}'\geq b_{jh}l_{h}'= b_{jh}>0$ and $l_j>0$, we have $l'_j>0$. It follows that
 \begin{gather*}
c_k=\Bigg(-l'_k +\sum\limits_{t=2}^{k-1}b_{kt}l_{t}'\Bigg)+0-b_{kj}+0= \sum\limits_{t=2}^{k-1}b_{kt}l_{t}' -b_{kj}\geq b_{kj}(l'_j-1)\geq0.
 \end{gather*}
\end{itemize}
Hence $c_k\geq0$ for all $k\in[2,n]$. The proof is completed.
\end{proof}

\begin{Remark}In the ordinary cluster algebras, the exchange relations are binomial relations, but the exchange relations of~the generalized cluster algebras are polynomial relations. Therefore we need more detailed discussions. For example, in the proof of~Lemma~\ref{lemim} we should choose the smallest integer $h\in[2,j-1]$ such that $b_{jh}l_h>0$ which implies $b_{jh}>0$, $l_h>0$ and $b_{jk}l_k=0$ for $k\in[2,h-1]$. The choice of~such $h$ makes the most significant difference between our proof and the proof in~\cite{bfz}.
\end{Remark}

In order to prove Theorem \ref{thmL=U} below, we need the following two lemmas.
\begin{Lemma}\label{lemzp}
Let~$\big(\widetilde{\mathbf{x}},\rho,\widetilde{B}\big)$ be acyclic and coprime.
For $j\in J=\{j\in[2,n]\,|\, b_{1j}=0\}$, if~$z\in \mathop{\mathbb{ZP}}\!\big[x^{\pm1}_2,\ldots,x^{\pm1}_n\big]$ and $zP_1\in\mathop{\mathbb{ZP}}\!\big[x^{\pm1}_{2},\ldots,x^{\pm1}_{j-1},x_j,x'_j,
x^{\pm1}_{j+1},\ldots,x^{\pm1}_{n}\big]$,
then
\begin{gather*}
z\in\mathop{\mathbb{ZP}}
\big[x^{\pm1}_{2},\ldots,x^{\pm1}_{j-1},x_j,x'_j,x^{\pm1}_{j+1},\ldots,x^{\pm1}_{n}\big].
\end{gather*}
\end{Lemma}
\begin{proof}
For each $j\in J$, we can write the element $z$ in the form $z=\sum\limits_{k=a_j}^{b_j}c_{j,k}x^{k}_j$ with $c_{j,k}\in \mathop{\mathbb{ZP}}\!\big[x^{\pm1}_{2},\ldots,x^{\pm1}_{j-1},x^{\pm1}_{j+1},\ldots,x^{\pm1}_{n}\big]$ and $c_{j,a_j}\neq0$.
If $a_j\geq0$, then we have that
\begin{gather*}
z\in \mathop{\mathbb{ZP}}\!\big[x^{\pm1}_{2},\ldots, x^{\pm1}_{j-1},x_j,x^{\pm1}_{j+1},\ldots,x^{\pm1}_{n}\big].
\end{gather*}
Now assume that $a_j<0$. The fact $zP_1\in\mathop{\rm Im}(f)$ implies that
\begin{gather*}
c_kx_{j}^{k}P_1= c_kP^{-|k|}_j(x'_j)^{|k|}P_1\in \mathop{\mathbb{ZP}}\!\big[x^{\pm1}_{2},\ldots,x^{\pm1}_{j-1},x_j,x'_j,x^{\pm1}_{j+1},\ldots,x^{\pm1}_{n}\big]
\end{gather*}
for $k<0$. Since $P_1$ and $P_j$ are coprime, we have that
\begin{gather*}
c_kP^{-|k|}_j\in \mathop{\mathbb{ZP}}\!\big[x^{\pm1}_{2},\ldots,x^{\pm1}_{j-1},x_j,x'_j,x^{\pm1}_{j+1},\ldots,x^{\pm1}_{n}\big].
\end{gather*}
The proof is completed.
\end{proof}

The following result shows that the upper bound of~the generalized cluster algebra of~rank 2 is equal to the corresponding lower bound.
\begin{Lemma}\label{lemn=2}
Suppose that $n=2$. If the generalized seed $\big(\widetilde{\mathbf{x}},\rho,\widetilde{B}\big)$ is coprime, then
\begin{gather}\label{n=2}
\mathop{\mathbb{ZP}}\!\big[x^{\pm1}_1,x_{2},x_{2}'\big]\cap\mathop{\mathbb{ZP}}\!\big[x_{1},x_{1}',x^{\pm1}_2\big] =\mathop{\mathbb{ZP}}[x_1,x_{1}',x_{2},x_{2}'].
\end{gather}
\end{Lemma}
\begin{proof}
Obviously, we have that $\mathop{\mathbb{ZP}}\!\big[x^{\pm1}_1,x_{2},x_{2}'\big]\cap\mathop{\mathbb{ZP}}\!\big[x_{1},x_{1}',x^{\pm1}_2\big] \supseteq\mathop{\mathbb{ZP}}[x_1,x_{1}',x_{2},x_{2}']$. Let~$y\in\mathop{\mathbb{ZP}}\!\big[x^{\pm1}_1,x_{2},x_{2}'\big]\cap\mathop{\mathbb{ZP}}\!\big[x_{1},x_{1}',x^{\pm1}_2\big]$. Note that the element $y$ can be written as $y=\sum\limits_{k=a}^{b}c_kx_{1}^{k}$ with $c_k\in\mathop{\mathbb{ZP}}\nolimits^{\rm st}[x_2,x'_2]$ and $c_a\neq0$. By Lemma~\ref{lemdirectsum} (3), if the leading term $\mathop{\rm LT}(y)=f(c_a)x^{a}_1\in\mathop{\mathbb{ZP}}\!\big[x_{1},x^{\pm1}_2\big]$, i.e., $a > 0$, then $y\in\mathop{\mathbb{ZP}}[x_1,x_{2},x_{2}']$.

Suppose that $a\leq0$. We will prove ($\ref{n=2}$) by induction on $-a$. It is enough to find an~element $y'\in\mathop{\mathbb{ZP}}[x_1,x'_1,x_2,x'_2]$ such that $\mathop{\rm LT}(y)=\mathop{\rm LT}(y')$. By Lemma~\ref{lemdivisible}, we have $f(c_a)P^{a}_1\in\mathop{\mathbb{ZP}}\big[x^{\pm1}_2\big]$, i.e., there exists some element $z\in\mathop{\mathbb{ZP}}\big[x^{\pm1}_2\big]$ such that $f(c_a)=zP^{-a}_1$. By Lemma~\ref{lemim}, if $b_{12}=0$ then $z\in\mathop{\mathbb{ZP}}\big[x^{\pm1}_2\big]=\mathop{\rm Im}(f)$ and if $b_{12}\neq0$ then $\mathop{\rm Im}(f)=\mathop{\mathbb{ZP}}[x_2,x'_2]$. Note that $P_1$ and $P_2$ are coprime. By applying Lemma~\ref{lemzp} repeatedly, we conclude that $z\in\mathop{\mathbb{ZP}}[x_2,x'_2]=\mathop{\rm Im}(f)$. It follows that there exists some element $z'\in \mathop{\mathbb{ZP}}[x_2,x'_2]$ such that $z=f(z')$. Let~$y'=z'(x'_1)^{-a}$, then we have that $y'\in\mathop{\mathbb{ZP}}[x_1,x'_1,x_2,x'_2]$ and
$
\mathop{\rm LT}(y')=f(z')P^{-a}_1x^{a}_1=zP^{-a}_1x^{a}_1=f(c_a)x^{a}_1=\mathop{\rm LT}(y).
$
Hence the absolute value of~the power of~$x_1$ in $y-y'$ is strictly less than $-a$. We see that $y-y'\in \mathop{\mathbb{ZP}}[x_1,x'_1,x_2,x'_2]$ by the induction hypothesis. Therefore $y\in \mathop{\mathbb{ZP}}[x_1,x'_1,x_2,x'_2]$, as~desired. The proof is completed.
\end{proof}

\begin{Theorem}\label{thmL=U}
If the generalized seed $\big(\widetilde{\mathbf{x}},\rho,\widetilde{B}\big)$ is coprime and acyclic, then
\begin{gather*}
\mathcal{L}\big(\widetilde{\mathbf{x}},\rho,\widetilde{B}\big)=\mathcal{U}\big(\widetilde{\mathbf{x}},\rho,\widetilde{B}\big).
\end{gather*}
\end{Theorem}
\begin{proof}
Let~$n$ be the rank of~$\mathcal{A}\big(\widetilde{\mathbf{x}},\rho,\widetilde{B}\big)$. We prove the statement by induction on $n$. We denote $\smash{\mathcal{U}:=\mathcal{U}\big(\widetilde{\mathbf{x}},\rho,\widetilde{B}\big)}$ and $\smash{\mathcal{L}:=\mathcal{L}\big(\widetilde{\mathbf{x}},\rho,\widetilde{B}\big)}$ for simplification purposes. When $n=1$, by using (\ref{equxxprime}), we have $\mathcal{L} =\mathcal{U}$. When $n=2$, the statement follows by Lemmas~\ref{lemup} and~\ref{lemn=2}. Now let $n\geq3$ and assume that the upper bound coincides with the lower bound if the rank of~the acyclic and coprime generalized cluster algebra is less than $n$. By Lemma~\ref{lemup}, we have that
\begin{gather*}
\mathcal{U}=\mathop{\mathbb{ZP}}\!\big[x_1,x'_1,x^{\pm1}_2,\ldots,x^{\pm1}_n\big]\cap\bigcap\limits_{i=2}^{n} \mathop{\mathbb{ZP}}\!\big[x^{\pm1}_1,\ldots,x^{\pm1}_{i-1},x_i,x'_i,x^{\pm1}_{i+1},\ldots,x^{\pm1}_n\big].
\end{gather*}
Let~the generalized seed $\big(\widetilde{\mathbf{x}}^{\dag},\rho^{\dag},\widetilde{B}^{\dag}\big)$ be the freezing of~$\big(\widetilde{\mathbf{x}},\rho,\widetilde{B}\big)$ at $x_1$. It follows that $\mathop{\mathbb{ZP}}\nolimits^{\dag}=\mathop{\mathbb{ZP}}\!\big[x^{\pm1}_1\big]$. For $i\in\big[2,n\big]$, we have that
\begin{gather*}
\mathop{\mathbb{ZP}}\nolimits^{\dag}\big[x^{\pm1}_2,\ldots,x^{\pm1}_{i-1},x_i,x'_i,x^{\pm1}_{i+1},\ldots,x^{\pm1}_n\big] =\mathop{\mathbb{ZP}}\!\big[x^{\pm1}_1,\ldots,x^{\pm1}_{i-1},x_i,x'_i,x^{\pm1}_{i+1},\ldots,x^{\pm1}_n\big].
\end{gather*}
By the induction hypothesis, we obtain that
\begin{gather*}
\bigcap\limits_{i=2}^{n} \mathop{\mathbb{ZP}}\nolimits^{\dag} \big[x^{\pm1}_2\ldots,x^{\pm1}_{i-1},x_i,x'_i,x^{\pm1}_{i+1},\ldots,x^{\pm1}_n\big] =\mathop{\mathbb{ZP}}\nolimits^{\dag}\big[x_2,x'_2,\ldots,x_n,x'_n\big].
\end{gather*}
Thus it is enough to prove that
\begin{gather}\label{casen}
\mathop{\mathbb{ZP}}\!\big[x^{\pm1}_1,x_{2},x_{2}',\ldots,x_{n},x_{n}'\big] \cap\mathop{\mathbb{ZP}}\!\big[x_{1},x_{1}',x^{\pm1}_2,\ldots,x^{\pm1}_n\big] =\mathop{\mathbb{ZP}}[x_1,x_{1}',\ldots,x_{n},x_{n}'].
\end{gather}
The proof of~(\ref{casen}) is quite similar to that given earlier for (\ref{n=2}). It is easy to see that the ``$\supseteq$" part holds. We only need to show the ``$\subseteq$'' part. Assume{\samepage
\begin{gather*}
y\in \mathop{\mathbb{ZP}}\!\big[x^{\pm1}_1,x_{2},x_{2}',\ldots,x_{n},x_{n}'\big] \cap\mathop{\mathbb{ZP}}\!\big[x_{1},x_{1}',x^{\pm1}_2,\ldots,x^{\pm1}_n\big].
\end{gather*}
We can write $y=\sum\limits_{k=a}^{b}c_kx_{1}^{k}$ with $c_k\in\mathop{\mathbb{ZP}}\nolimits^{\rm st} [x_2,x'_2,\ldots,x_n,x'_n]$ and $c_a\neq0$.}

If $\mathop{\rm LT}(y)\in\mathop{\mathbb{ZP}}\big[x_{1},x^{\pm1}_2,\ldots,x^{\pm1}_n\big]$,
we have $y\in\mathop{\mathbb{ZP}}[x_1,x_{2},x_{2}',\ldots,x_{n},x_{n}']$ by Lemma~\ref{lemdirectsum}(3) and the statement is true. Hence we can assume that $a<0$. We prove (\ref{casen}) by induction on $|a|$. It suffices to prove that there exists some element $y_1\in\mathop{\mathbb{ZP}}\!\big[x_1,x_{1}',x_{2},x_{2}',\ldots,x_{n},x_{n}'\big]$ such that $\mathop{\rm LT}(y)=\mathop{\rm LT}(y_1)$. Set $z=f(c_a)P_{1}^{a}$. By Lemma~\ref{lemdivisible}, we have that $z=f(c_a)P_{1}^{a}\in\mathop{\mathbb{ZP}}\!\big[x^{\pm1}_2,\ldots,x^{\pm1}_n\big]$. Note that we have $zP^{-a}_1\in\mathop{\rm Im}(f)$.
We claim that $z\in \mathop{\rm Im}(f)$. By Lemma~\ref{lemim}, if $J=\varnothing$, then the claim is true. Suppose that $J\neq\varnothing$. By applying Lemma~\ref{lemzp} repeatedly, we~conclude that
$z\in\bigcap\limits_{j\in J} \mathop{\mathbb{ZP}}\!\big[x^{\pm1}_{2},\ldots,x^{\pm1}_{j-1},x_{j},x'_j,x^{\pm1}_{j+1}, \ldots,x^{\pm1}_{n}\big].$
It remains to show that
\begin{gather*}
\bigcap\limits_{j\in J} \mathop{\mathbb{ZP}}\!\big[x^{\pm1}_{2},\ldots,x^{\pm1}_{j-1},x_{j},x'_j,x^{\pm1}_{j+1},\ldots,x^{\pm1}_{n}\big] =\mathop{\rm Im}(f).
\end{gather*}
We obtain the generalized seed $\big(\widetilde{\mathbf{y}},\varepsilon,\widetilde{A}\big)$ from $\big(\widetilde{\mathbf{x}},\rho,\widetilde{B}\big)$ through freezing at $\{x_j\,|\, j\in [2,n]-J\}$ and then removing the cluster variable $x_1$. The group ring of~coefficients of~$\smash{\mathcal{A}\big(\widetilde{\mathbf{y}},\varepsilon,\widetilde{A}\big)}$ is
$\mathop{\mathbb{ZP}}\!\big[x^{\pm1}_j\,|\, j\in[2,n]-J\big]$.
By the induction hypothesis, we have that $\smash{\mathcal{L}\big(\widetilde{\mathbf{y}},\varepsilon,\widetilde{A}\big) =\mathcal{U}\big(\widetilde{\mathbf{y}},\varepsilon,\widetilde{A}\big)}$, namely
\begin{gather*}
\bigcap\limits_{j\in J}\mathop{\mathbb{ZP}}\big[x^{\pm1}_{2},\ldots,x^{\pm1}_{j-1},x_{j},x'_j,x^{\pm1}_{j+1},\ldots,x^{\pm1}_{n}\big] =\mathop{\mathbb{ZP}}\big[x_2,\widehat{x}_2,\ldots,x_n,\widehat{x}_{n}\big]=\mathop{\rm Im}(f).
\end{gather*}
Thus $z\in\mathop{\rm Im}(f)$ and there exists some $z_1\in\mathop{\mathbb{ZP}}[x_2,x'_2,\ldots,x_n,x'_n]$ such that $z=f(z_1)$. Let~$y_1=z_1(x'_1)^{|a|}$. Then $\mathop{\rm LT}(y_1)=\mathop{\rm LT}(y)$. This completes the proof.
\end{proof}

The following result follows immediately from Theorems~\ref{thmlinind} and~\ref{thmL=U}.
\begin{Corollary}
If the generalized seed $\big(\widetilde{\mathbf{x}},\rho,\widetilde{B}\big)$ is acyclic and coprime, then the standard monomials in $\{x_1,x'_1,\ldots,x_n,x'_n\}$ form a $\mathop{\mathbb{ZP}}$-basis of~$\mathcal{A}\big(\widetilde{\mathbf{x}},\rho,\widetilde{B}\big)$.
\end{Corollary}

For convenience, the basis consisting of~the standard monomials is called the standard monomial basis.

\appendix
\section{Generalized upper cluster algebras}\label{appendix}

In this appendix, we prove that every acyclic generalized cluster algebra coincides with its corresponding upper cluster algebra without the assumption of~the existence of~coprimality. We mimic the proof of~the main results given in \cite{muller14}.

\begin{Definition}
Let~$\big(\widetilde{\mathbf{x}},\rho,\widetilde{B}\big)$ be a generalized seed. The generalized upper cluster algebra of~$\mathcal{A}\smash{\big(\widetilde{\mathbf{x}},\rho,\widetilde{B}\big)}$ is defined as
\begin{gather*}
\widetilde{\mathcal{U}}\big(\widetilde{\mathbf{x}},\rho,\widetilde{B}\big):= \bigcap\limits_{\left(\widetilde{\mathbf{y}},\varepsilon,\widetilde{A}\right)\sim \left(\widetilde{\mathbf{x}},\rho,\widetilde{B}\right)} \mathop{\mathbb{ZP}}\!\big[y^{\pm1}_1,\ldots,y^{\pm1}_n\big],
\end{gather*}
where the generalized seed $\big(\widetilde{\mathbf{y}},\varepsilon,\widetilde{A}\big)$ is mutation-equivalent to $\big(\widetilde{\mathbf{x}},\rho,\widetilde{B}\big)$.

\end{Definition}
Obviously, the generalized upper cluster algebra $\widetilde{\mathcal{U}}\big(\widetilde{\mathbf{x}},\rho,\widetilde{B}\big)$ is contained in the corresponding upper bound.

\begin{Lemma}\label{lemA+subsetU+}
Let~$\{x_{i_1},\ldots,x_{i_k}\}\subsetneq \mathbf{x}$. If $\big(\widetilde{\mathbf{x}}^{\dag},\rho^{\dag},\widetilde{B}^{\dag}\big)$ is the freezing of~$\big(\widetilde{\mathbf{x}},\rho,\widetilde{B}\big)$ at the set $\{x_{i_1},\ldots,x_{i_k}\}$, then we have that
\begin{gather*}
\mathcal{A}\big(\widetilde{\mathbf{x}}^{\dag},\rho^{\dag},\widetilde{B}^{\dag}\big) \subseteq \mathcal{A}\big(\widetilde{\mathbf{x}},\rho,\widetilde{B}\big)\big[(x_{i_1}\cdots x_{i_k})^{-1}\big] \subseteq \mathcal{\widetilde{U}}\big(\widetilde{\mathbf{x}},\rho,\widetilde{B}\big)\big[(x_{i_1}\cdots x_{i_k})^{-1}\big] \subseteq \mathcal{\widetilde{U}}\big(\widetilde{\mathbf{x}}^{\dag},\rho^{\dag},\widetilde{B}^{\dag}\big).
\end{gather*}
\end{Lemma}
\begin{proof}
By the definitions of~the freezing of~generalized cluster algebras and the generalized upper cluster algebras, the first and the third inclusions are immediate. By the Laurent phenomenon, i.e., $\mathcal{A}\smash{\big(\widetilde{\mathbf{x}},\rho,\widetilde{B}\big)\subseteq \mathcal{\widetilde{U}}\big(\widetilde{\mathbf{x}},\rho,\widetilde{B}\big)}$, we obtain the second inclusion.
\end{proof}

\begin{Definition}
Let~$\big(\widetilde{\mathbf{x}}^{\dag},\rho^{\dag},\widetilde{B}^{\dag}\big)$ be the freezing of~$\big(\widetilde{\mathbf{x}},\rho,\widetilde{B}\big)$ at $\{x_{i_1},\ldots,x_{i_k}\}$. If
\begin{gather*}
\mathcal{A}\big(\widetilde{\mathbf{x}}^{\dag},\rho^{\dag},\widetilde{B}^{\dag}\big)= \mathcal{A}\big(\widetilde{\mathbf{x}},\rho,\widetilde{B}\big)\big[(x_{i_1}\cdots x_{i_k})^{-1}\big],
\end{gather*}
then we call $\mathcal{A}\big(\widetilde{\mathbf{x}}^{\dag},\rho^{\dag},\widetilde{B}^{\dag}\big)$ a cluster localization of~$\mathcal{A}\big(\widetilde{\mathbf{x}},\rho,\widetilde{B}\big)$.
\end{Definition}
If $\mathcal{A}\big(\widetilde{\mathbf{x}}^\dag,\rho^\dag,\widetilde{B}^\dag\big)$ is a cluster localization of~$\mathcal{A}\big(\widetilde{\mathbf{x}},\rho,\widetilde{B}\big)$ and $\mathcal{A}\big(\widetilde{\mathbf{x}}^\ddag,\rho^\ddag,\widetilde{B}^\ddag\big)$ a cluster localization of~$\smash{\mathcal{A}\big(\widetilde{\mathbf{x}}^\dag,\rho^\dag,\widetilde{B}^\dag\big)}$, then $\smash{\mathcal{A}\big(\widetilde{\mathbf{x}}^\ddag,\rho^\ddag,\widetilde{B}^\ddag\big)}$ is also a cluster localization of~$\smash{\mathcal{A}\big(\widetilde{\mathbf{x}},\rho,\widetilde{B}\big)}$.

\begin{Definition}
Let~$\mathcal{A}\big(\widetilde{\mathbf{x}},\rho,\widetilde{B}\big)$ be a generalized cluster algebra. Let~$\{\mathcal{A}_i\,|\, i\in I\}$ be the set such that $\mathcal{A}_i$ are the cluster localizations of~$\smash{\mathcal{A}\big(\widetilde{\mathbf{x}},\rho,\widetilde{B}\big)}$. For each prime ideal $\mathfrak{p}$ of~$\smash{\mathcal{A}\big(\widetilde{\mathbf{x}},\rho,\widetilde{B}\big)}$, if there exists some $i\in I$ such that $\mathfrak{p}\mathcal{A}_i\subsetneq \mathcal{A}_i$, then the set $\{\mathcal{A}_i\,|\, i\in I\}$ is called a cover of~$\mathcal{A}\smash{\big(\widetilde{\mathbf{x}},\rho,\widetilde{B}\big)}$.
\end{Definition}

\begin{Example}
Let~$(\mathbf{x},\rho,B)$ be the generalized seed from Example \ref{example1}. The freezing of~$\mathcal{A}(\mathbf{x},\rho,B)$ at $x_1$ is $
\mathcal{A}_1:=\mathbb{Z}\big[x^{\pm1}_1,x_2,\frac{x_1+1}{x_2}\big]=\mathcal{A}(\mathbf{x},\rho,B)\big[x^{-1}_1\big],
$ and the freezing at $x_2$ is $
\mathcal{A}_2:=\mathbb{Z}\big[x_1,x^{\pm1}_2,\frac{1+hx_2+x^{2}_2}{x_1}\big]=\mathcal{A}(\mathbf{x},\rho,B)\big[x^{-1}_2\big]$.
Since the ideal $(x_1,x_2)=\mathcal{A}(\mathbf{x},\rho,B)$, we conclude that $\{x_1,x_2\}\nsubseteq\mathfrak{p}$ for any prime ideal $\mathfrak{p}$ of~$\mathcal{A}(\mathbf{x},\rho,B)$. It follows that $\{\mathcal{A}_1,\mathcal{A}_2\}$ is a cover of~$\mathcal{A}(\mathbf{x},\rho,B)$.
\end{Example}

\begin{Lemma}\label{lemcovertransitive}
If $\mathcal{A}\big(\widetilde{\mathbf{x}},\rho,\widetilde{B}\big)$ has a cover $\{\mathcal{A}_i\,|\, i\in I\}$ and each $\mathcal{A}_i$ has a cover $\{\mathcal{A}_{ij}\,|\, j\in J_i\}$, then $\{\mathcal{A}_{ij}\,|\, i\in I,j\in J_i\}$ is a cover of~$\mathcal{A}\smash{\big(\widetilde{\mathbf{x}},\rho,\widetilde{B}\big)}$. Namely, the covers are transitive.
\end{Lemma}
\begin{proof}
Let~$\mathfrak{p}$ be a prime ideal of~$\mathcal{A}\big(\widetilde{\mathbf{x}},\rho,\widetilde{B}\big)$, and suppose there exists some $\mathcal{A}_i$ such that \mbox{$\mathfrak{p}\mathcal{A}_i\subsetneq \mathcal{A}_i$}. Since $\mathfrak{p}\mathcal{A}_i$ is a proper ideal of~$\mathcal{A}_i$, there exists a maximal ideal $\mathfrak{m}$ in $\mathcal{A}_i$ such that $\mathfrak{p}\mathcal{A}_i\subseteq \mathfrak{m}$. There exists some $A_{ij}$ such that $\mathfrak{m}\mathcal{A}_{ij}\subsetneq\mathcal{A}_{ij}$. Therefore we have that $\mathfrak{p}\mathcal{A}_{ij} \subseteq$ $\mathfrak{m}\mathcal{A}_{ij}\subsetneq\mathcal{A}_{ij}$. This completes the proof.
\end{proof}

\begin{Lemma}\label{lemcover}
Let~$\{\mathcal{A}_i\,|\, i\in I\}$ be a cover of~the generalized cluster algebra $\mathcal{A}\big(\widetilde{\mathbf{x}},\rho,\widetilde{B}\big)$. Let~$\mathcal{\widetilde{U}}_i$ denote the generalized upper cluster algebra of~$\mathcal{A}_i$ for each $i\in I$. We have
\begin{enumerate}\itemsep=0pt
 \item[{\rm (1)}]
$\mathcal{A}\big(\widetilde{\mathbf{x}},\rho,\widetilde{B}\big)=\bigcap\limits_{i\in I}\mathcal{A}_i;$

 \item[{\rm (2)}] if $\mathcal{A}_i=\mathcal{\widetilde{U}}_i$ for all $i\in I$, then $\mathcal{A}\big(\widetilde{\mathbf{x}},\rho,\widetilde{B}\big) =\mathcal{\widetilde{U}}\big(\widetilde{\mathbf{x}},\rho,\widetilde{B}\big)$.
\end{enumerate}
\end{Lemma}
\begin{proof}
The proof of~the lemma is quite similar to the one used in \cite[Proposition~2, Lemma~2]{muller14}, so is omitted.
\end{proof}

A~generalized cluster algebra $\mathcal{A}\big(\widetilde{\mathbf{x}},\rho,\widetilde{B}\big)$ is called isolated if the principal part $B=0$.

\begin{Proposition}\label{propisolated}
If $\mathcal{A}\big(\widetilde{\mathbf{x}},\rho,\widetilde{B}\big)$ is an~isolated generalized cluster algebra, then we have that
$\smash{\mathcal{A}\big(\widetilde{\mathbf{x}},\rho,\widetilde{B}\big)
=\mathcal{\widetilde{U}}\big(\widetilde{\mathbf{x}},\rho,\widetilde{B}\big)}$.
\end{Proposition}
\begin{proof}
Proving the proposition uses the same ideas in \cite[Proposition~3]{muller14}.
\end{proof}

\begin{Proposition}\label{propcover}
If the generalized cluster algebra $\mathcal{A}\big(\widetilde{\mathbf{x}},\rho,\widetilde{B}\big)$ is acyclic, then it has a cover $\{\mathcal{A}_i\,|\, i\in I\}$ with each $\mathcal{A}_i$ being isolated.
\end{Proposition}
\begin{proof}
We can assume that $b_{ij}\geq0$ for $i,j\in[1,n]$ with $i>j$.
There exist $k,l\in[1,n]$ such that~$l$ is a source in $\smash{\Gamma\big(\widetilde{\mathbf{x}},\rho,\widetilde{B}\big)}$ and $b_{kl}<0$. Recall the exchange relation
\begin{gather*}
x_lx'_l=\sum\limits_{r=0}^{d_l-1}\rho_{l,r}\prod\limits_{i=1}^{m}x_i^{r[\beta_{il}]_+ +(d_l-r)[-\beta_{il}]_+} +\prod\limits_{i=n+1}^{m}x_i^{[b_{il}]_+}.
\end{gather*}
It follows that
\begin{gather*}
1=x_lx'_l\prod\limits_{i=n+1}^{m}x_i^{-[b_{il}]_+} -\sum\limits_{r=0}^{d_l-1}\rho_{l,r}\prod\limits_{i=1}^{m}x_i^{r[\beta_{il}]_+ +(d_l-r)[-\beta_{il}]_+} \prod\limits_{i=n+1}^{m}x_i^{-[b_{il}]_+}.
\end{gather*}
Since $\beta_{kl}<0$, the variables $x_k$ and $x_l$ appear in the right hand side of~the above equation. Thus the ideal $(x_k,x_l)$ is equal to $\mathcal{A}\smash{\big(\widetilde{\mathbf{x}},\rho,\widetilde{B}\big)}$. The remainder of~the proof follows the one in~\mbox{\cite[Proposition~4]{muller14}}.
\end{proof}

\begin{Corollary}
If the generalized cluster algebra $\mathcal{A}\big(\widetilde{\mathbf{x}},\rho,\widetilde{B}\big)$ is acyclic, then we have that
$
\mathcal{A}\big(\widetilde{\mathbf{x}},\rho,\widetilde{B}\big)=\mathcal{\widetilde{U}}\big(\widetilde{\mathbf{x}},\rho,\widetilde{B}\big).
$
\end{Corollary}

\subsection*{Acknowledgements}

The authors are greatly indebted to referees for their valuable comments and recommendations which definitely help to improve the readability and quality of~the paper. Liqian Bai was supported by NSF of~China (No.~11801445), the Natural Science Foundation of~Shaanxi Province (No.~2020JQ-116) and the Fundamental Research Funds for the Central Universities (No.~310201911cx027), Ming Ding was supported by NSF of~China (No.~11771217) and Fan Xu was supported by NSF of~China (No. 11471177).

\pdfbookmark[1]{References}{ref}
\LastPageEnding


\begin{thebibliography}{99}
\footnotesize\itemsep=0pt

\bibitem{BMRS}
Benito A., Muller G., Rajchgot J., Smith K.E., Singularities of locally acyclic
 cluster algebras, \href{https://doi.org/10.2140/ant.2015.9.913}{\textit{Algebra Number Theory}} \textbf{9} (2015), 913--936,
 \href{https://arxiv.org/abs/1404.4399}{arXiv:1404.4399}.

\bibitem{bfz}
Berenstein A., Fomin S., Zelevinsky A., Cluster algebras. {III}.~{U}pper bounds
 and double {B}ruhat cells, \href{https://doi.org/10.1215/S0012-7094-04-12611-9}{\textit{Duke Math.~J.}} \textbf{126} (2005), 1--52,
 \href{https://arxiv.org/abs/math.RT/0305434}{arXiv:math.RT/0305434}.

\bibitem{BMS}
Bucher E., Machacek J., Shapiro M., Upper cluster algebras and choice of ground
 ring, \href{https://doi.org/10.1007/s11425-018-9486-6}{\textit{Sci. China Math.}} \textbf{62} (2019), 1257--1266,
 \href{https://arxiv.org/abs/1802.04835}{arXiv:1802.04835}.

\bibitem{CS}
Chekhov L., Shapiro M., Teichm\"uller spaces of {R}iemann surfaces with
 orbifold points of arbitrary order and cluster variables, \href{https://doi.org/10.1093/imrn/rnt016}{\textit{Int. Math.
 Res. Not.}} \textbf{2014} (2014), 2746--2772, \href{https://arxiv.org/abs/1111.3963}{arXiv:1111.3963}.

\bibitem{ca1}
Fomin S., Zelevinsky A., Cluster algebras. {I}.~{F}oundations, \href{https://doi.org/10.1090/S0894-0347-01-00385-X}{\textit{J.~Amer.
 Math. Soc.}} \textbf{15} (2002), 497--529, \href{https://arxiv.org/abs/math.RT/0104151}{arXiv:math.RT/0104151}.

\bibitem{ca2}
Fomin S., Zelevinsky A., Cluster algebras. {II}.~{F}inite type classification,
 \href{https://doi.org/10.1007/s00222-003-0302-y}{\textit{Invent. Math.}} \textbf{154} (2003), 63--121,
 \href{https://arxiv.org/abs/math.RA/0208229}{arXiv:math.RA/0208229}.

\bibitem{gsv18}
Gekhtman M., Shapiro M., Vainshtein A., Drinfeld double of {${\rm GL}_n$} and
 generalized cluster structures, \href{https://doi.org/10.1112/plms.12086}{\textit{Proc. Lond. Math. Soc.}} \textbf{116}
 (2018), 429--484, \href{https://arxiv.org/abs/1605.05705}{arXiv:1605.05705}.

\bibitem{GHK1}
Gross M., Hacking P., Keel S., Birational geometry of cluster algebras,
 \href{https://doi.org/10.14231/AG-2015-007}{\textit{Algebr. Geom.}} \textbf{2} (2015), 137--175, \href{https://arxiv.org/abs/1309.2573}{arXiv:1309.2573}.

\bibitem{LLM}
Lee K., Li L., Mills M.R., A combinatorial formula for certain elements of
 upper cluster algebras, \href{https://doi.org/10.3842/SIGMA.2015.049}{\textit{SIGMA}} \textbf{11} (2015), 049, 24~pages,
 \href{https://arxiv.org/abs/1409.8177}{arXiv:1409.8177}.

\bibitem{MM15}
Matherne J.P., Muller G., Computing upper cluster algebras, \href{https://doi.org/10.1093/imrn/rnu027}{\textit{Int. Math.
 Res. Not.}} \textbf{2015} (2015), 3121--3149, \href{https://arxiv.org/abs/1307.0579}{arXiv:1307.0579}.

\bibitem{muller13}
Muller G., Locally acyclic cluster algebras, \href{https://doi.org/10.1016/j.aim.2012.10.002}{\textit{Adv. Math.}} \textbf{233}
 (2013), 207--247, \href{https://arxiv.org/abs/1111.4468}{arXiv:1111.4468}.

\bibitem{muller14}
Muller G., {${\mathcal A}={\mathcal U}$} for locally acyclic cluster algebras,
 \href{https://doi.org/10.3842/SIGMA.2014.094}{\textit{SIGMA}} \textbf{10} (2014), 094, 8~pages, \href{https://arxiv.org/abs/1308.1141}{arXiv:1308.1141}.

\bibitem{nak2}
Nakanishi T., Structure of seeds in generalized cluster algebras,
 \href{https://doi.org/10.2140/pjm.2015.277.201}{\textit{Pacific~J. Math.}} \textbf{277} (2015), 201--217, \href{https://arxiv.org/abs/1409.5967}{arXiv:1409.5967}.

\bibitem{P}
Plamondon P.-G., Generic bases for cluster algebras from the cluster category,
 \href{https://doi.org/10.1093/imrn/rns102}{\textit{Int. Math. Res. Not.}} \textbf{2013} (2013), 2368--2420,
 \href{https://arxiv.org/abs/1111.4431}{arXiv:1111.4431}.

\end{thebibliography}
\end{document}